\documentclass[12pt]{amsart}
\usepackage{amsmath}
\usepackage{hyperref}
\usepackage{amssymb}
\usepackage{latexsym}
\usepackage{amscd}
\usepackage{graphicx} 
\usepackage{amsthm}
\usepackage{mathrsfs}
\usepackage{xypic}
\usepackage{bm}

\newdimen\AAdi%
\newbox\AAbo%
\def\AAk#1#2{\s_etbox\AAbo=\hbox{#2}\AAdi=\wd\AAbo\kern#1\AAdi{}}%
\def\AAr#1#2#3{\s_etbox\AAbo=\hbox{#2}\AAdi=\ht\AAbo\raise#1\AAdi\hbox{#3}}%
\font\tenmsb=msbm10 at 12pt \font\sevenmsb=msbm7 at 8pt
\font\fivemsb=msbm5 at 6pt
\newfam\msbfam
\textfont\msbfam=\tenmsb \scriptfont\msbfam=\sevenmsb
\scriptscriptfont\msbfam=\fivemsb
\def\Bbb#1{{\tenmsb\fam\msbfam#1}}
\newtheorem{thm}{Theorem}[section]
\newtheorem{lem}{Lemma}[section]
\newtheorem{cor}{Corollary}[section]
\newtheorem{rem}{Remark}[section]
\newtheorem{pro}{Proposition}[section]
\newtheorem{defi}{Definition}[section]

\newcommand{\ba}{\begin{array}}
\newcommand{\ea}{\end{array}}

\newcommand{\Section}[2]{\setcounter{equation}{0}
\allowdisplaybreaks
\section[#1]{#2}}

\def\n{\nabla}

\def\ir#1{\mathbb R^{#1}}

\def\f#1#2{\frac{#1}{#2}}

\def\grs#1#2{\bold G_{#1,#2}}

\def\mc#1{\mathcal{#1}}
\def\nA#1{\big|\nabla|A^{#1}|^2\big|}

\def\pf#1{\frac{\partial}{\partial #1}}
\def\pd#1#2{\frac {\partial #1}{\partial #2}}
\def\ppd#1#2{\frac {\partial^2 #1}{\partial #2^2}}

\def\td{\tilde}
\def\ul{\underline}

\def\a{\alpha}
\def\be{\beta}

\def\p#1{\partial #1}

\def\de{\delta}
\def\De{\Delta}

\def\ep{\varepsilon}

\def\G{\Gamma}
\def\g{\gamma}

\def\la{\lambda}

\def\om{\omega}
\def\Om{\Omega}
\def\th{\theta}

\def\ul{\underline}
\def\w{\wedge}

\def\R{\Bbb{R}}
\def\C{\Bbb{C}}
\def\tr{\mbox{tr}}

\def\lan{\langle}
\def\ran{\rangle}
\def\ra{\rightarrow}

\def\ze{\zeta}

\subjclass{58E20,53A10.}

\begin{document}
\title
[Curvature estimates] {Curvature estimates for minimal submanifolds of higher codimension and small G-rank}

\author
[J. Jost, Y. L. Xin and Ling Yang]{J. Jost, Y. L. Xin and Ling Yang}
\address{Max Planck Institute for Mathematics in the
Sciences, Inselstr. 22, 04103 Leipzig, Germany, and \\
Department of Mathematics, University of Leipzig, 04081 Leipzig, Germany.}
\email{jost@mis.mpg.de}
\address {Institute of Mathematics, Fudan University,
Shanghai 200433, China.} \email{ylxin@fudan.edu.cn}
\address{Institute of Mathematics, Fudan University,
Shanghai 200433, China.} \email{yanglingfd@fudan.edu.cn}
\thanks{The first author is supported by the ERC Advanced Grant
  FP7-267087. The second  author and the third  author are supported partially by NSFC. They are also grateful to the Max Planck
Institute for Mathematics in the Sciences in Leipzig for its
hospitality and  continuous support. }

\begin{abstract}
We obtain new curvature estimates and Bernstein type results for
minimal $n-$submanifolds in $\ir{n+m},\, m\ge 2$ under the condition
that  the rank of its Gauss map is at most 2. In particular, this applies to minimal
surfaces in Euclidean spaces of arbitrary codimension.
\end{abstract}
\maketitle

\Section{Introduction}{Introduction}

The classical Bernstein theorem says that an entire minimal graph in
$\ir{3}$ has to be an affine plane. The mathematics behind this
theorem has proved to be enormously rich. It connects with
differential geometry, partial differential equations, and complex
analysis, and it has been a stimulus for important developments in
all these fields (for some references, see for instance \cite{x}).
In particular, the question emerged and has been intensively
investigated to what extent this result can be generalized in
various directions, that is, under which conditions a minimal
submanifold of some Euclidean space (or a sphere) is necessarily
affine linear (or a sub-sphere).

In particular, Bernstein type theorems for higher dimension and
codimension have been studied. Thus, let $M\to\ir{n+m}$ be an $n$
dimensional submanifold in Euclidean space $\ir{n+m}$. In recent
work (\cite{j-x}, \cite{j-x-y2} and \cite{j-x-y3}), we have
systematically used  geometric properties of  Grassmannian manifolds
and the regularity theory of harmonic maps to obtain new Bernstein
type results for higher dimension $n\ge 3$ and codimension $m\ge 2$.
The key point is that the Gauss map of such a minimal submanifold is
a harmonic map with values in a Grassmann manifold. Thus, our
approach combines methods from differential geometry and partial
differential equations. This leads us to the question whether this
can also be combined with the complex analysis approach. The complex
analysis approach is, of course, naturally restricted to the case
$n=2$, if, for the sake of the discussion, we ignore such issues as
subvarieties of complex spaces.

Thus, the present paper is concerned with the case $n=2$ and $m\ge
2$. Now, the target manifold of the Gauss map is $\grs{2}{m}$, the
complex quadric, and the Gauss map is holomorphic. A powerful
traditional approach to this problem investigates the  value
distribution of the Gauss image within the framework of complex
geometry. This was started by  Chern and Osserman \cite{c-o}. From
their results, an analogue of Moser's Bernstein theorem \cite{m}
that works for $n\ge 2$ and $m=1$ also holds for the case $n=2$ and
$m\ge 2$. More precisely, for an entire minimal graph given by
$f:\ir{2}\to\ir{m}$ if
$\De_f=[\det(\de_{ij}+f^\a_if^\a_j)]^{\f{1}{2}}$ is uniformly
bounded, then $f$ is affine linear, and thus represents an affine
plane in $\ir{2+m}$.

In the present paper, we use curvature estimate techniques to
improve this result. This will also enable us to achieve some
technical generalization which we shall now formulate. For a minimal
$n-$submanifold $M$ in $\ir{n+m}$ we consider the rank of the Gauss
map, which is   called  the $G-rank$ for simplicity. Our
condition then simply is $G-rank\le 2$.  Obviously, this class of submanifolds contains surfaces
in $\ir{2+m}$, as well as cylinders over surfaces in  $\ir{3}$.  In
\cite{d-f}, Dajczer and Florit gave a parametric description of all
Euclidean minimal submanifolds of $G-rank = 2$. In particular, they
showed  that
complete minimal submanifolds with $G-rank=2$ have dimension $n=3$ at most (without Euclidean factor).

Here, then, are our  main results.

\begin{thm}
Let $M$ be an $n$-dimensional complete minimal submanifold in
$\R^{n+m}$ with $G-rank\le 2$ and positive $w$-function (see (\ref{wf}).
If $M$ has polynomial volume growth and the function $v=w^{-1}$  has growth
\begin{equation}
\max_{D_R(p_0)}v=o(R^{\f{2}{3}})
\end{equation}
for a fixed point $p_0$, then $M$ has to be an affine linear subspace.
\end{thm}

Then, we have
\begin{thm}\label{t3}
Let $M=\{(x,f(x)):x\in \R^n\}$ be an entire minimal graph given by a vector-valued function $f:\R^n\ra \R^m$ with  $G-rank\le 2$. If the slope of $f$
satisfies
\begin{equation}
\De_f=\left[\det\Big(\de_{ij}+\sum_{\a}\f{\p f^\a}{\p x^i}\f{\p f^\a}{\p x^j}\Big)\right]^{\f{1}{2}}=o(R^{\f{2}{3}}),
\end{equation}
where $R^2=|x|^2+|f(x)|^2$,
then $f$ has to be an affine linear function.
\end{thm}

\begin{thm}\label{t4}
Let $f:\R^2\ra \R^m$ $(x^1,x^2)\mapsto (f^1,\cdots,f^m)$ be an entire
solution of the minimal surface system
\begin{equation}\label{mi}
\Big(1+\Big|\f{\p f}{\p x^2}\Big|^2\Big)\f{\p^2 f}{(\p x^1)^2}-2\Big\lan \pd{f}{x^1},\pd{f}{x^2}\Big\ran\f{\p^2 f}{\p x^1\p x^2}+\Big(1+\Big|\pd{f}{x^1}\Big|^2\Big)\f{\p^2 f}{(\p x^2)^2}=0.
\end{equation}
If there exists $\ep>0$,
\begin{equation}
\De_f=\det\Big(\de_{ij}+\sum_\a \pd{f^\a}{x^i}\pd{f^\a}{x^j}\Big)^{\f{1}{2}}=O(R^{1-\ep})
\end{equation}
with $R=|x|$, then $f$ has to be affine linear.
\end{thm}

This is an improvement of the Chern-Osserman theorem mentioned above
which required $\De_f$ to be  bounded.

The paper is organized as follows. After \S 2 on basic notation and
formulae we describe the special features of the $G-rank\le 2$ case
in \S 3. Then in \S 4 , using  subharmonic functions obtained from
the geometry of the Grassmann manifolds and a Bochner type formula
for the squared norm of the second fundamental form $|B|^2$ we can
obtain $L^p-$estimates and point-wise estimates for $|B|^2$. Those
estimates lead to Bernstein type results.  \S 5 is devoted to the
graphic situation. In the final section we discuss the sharpness of
our estimates. We find that holomorphic curves reach all the
possible equalities in all the geometric inequalities (\ref{b4}),
(\ref{dew}) and (\ref{kato}).

We thank Marcos Dajczer for informing us about \cite{d-f}.

\Section{Fundamental formulas}{Fundamental formulas}\label{s1}

Let $M$ be an $n$-dimensional submanifold in an $(n+m)$-dimensional Riemannian manifold $\bar{M}$  with second fundamental form
$B.$ Let $\bar{\n}$ denote
the Levi-Civita connection on $\bar{M}$. It naturally induces connections
on the tangent bundle $TM$, the normal bundle $NM$ and various induced bundles over $M.$ For notational simplicity  all of them are
denoted by $\n$. For arbitrary $\nu\in \G(NM)$ the shape operator $A^\nu: TM\ra TM$ satisfies
$\lan B_{X,Y},\nu\ran=\lan A^\nu(X),Y\ran$
for every $X,Y\in \G(TM)$.  It is self adjoint in the tangent spaces
of $M.$ The mean curvature field $H$ is defined to be the trace of the
second fundamental form. $M$ is called a \textit{minimal submanifold}
whenever $H$ vanishes on $M$ everywhere. The second fundamental form,
the curvature tensor of the submanifold, the curvature tensor of the
normal bundle and that of the ambient manifold are connected by the Gauss equations, the Codazzi equations and the Ricci equations (see \cite{x}, \S 1.1).

In this paper  we consider a minimal submanifold $M$ of dimension $n$ in the Euclidean space $\R^{n+m}$ with codimension $m\geq 2$.

Now, there is  an important tool, the  Gauss map. The Gauss map $\g:
M\ra \grs{n}{m}$ is defined by
$$\g(p)=T_p M\in \grs{n}{m}$$
via the parallel translation in $\R^{n+m}$ for every $p\in M$, where $\grs{n}{m}$ is the Grassmann manifold
consisting of the oriented linear $n$-subspaces in $\R^{n+m}$. One can write
$$\g(p)=e_1\w \cdots\w e_n$$
by using Pl\"{u}cker coordinates.  Here and in the sequel, $\{e_i\}$
is a local orthonormal tangent frame field of $M$ and $\{\nu_\a\}$
denotes a local orthonormal normal frame field of $M$;  we use the
summation convention and agree on the following ranges of indices:
$$1\leq i,j,k\leq n;\qquad 1\leq \a,\be,\g\leq m.$$

Let
$$h_{\a,ij}:=\lan B_{e_i e_j},\nu_\a\ran$$
be the coefficients of the second fundamental form $B$ of $M$ in $\R^{n+m}.$ Then,
\begin{equation}\label{Gm1}
\g_* e_i=h_{\a,ij}e_{j\a},
\end{equation}
where $e_{j\a}$ is obtained by replacing $e_j$ by $\nu_\a$ in $e_1\w\cdots\w e_n$.
The energy density of the Gauss map thus is nothing but the squared
norm of the second fundamental form (see \cite{x} \S
3.1),
$$e(\g)=\f{1}{2}\lan \g_*e_i,\g_* e_i\ran=\f{1}{2}\sum_{\a,i,j}h_{\a,ij}^2=\f{1}{2}|B|^2.$$

Given two unit $n$-vectors
$$A=a_1\w \cdots\w a_n,\qquad B=b_1\w \cdots\w b_n$$
in the Grassmann manifold $\grs{n}{m}$, their inner product is defined by
\begin{equation}
\lan A,B\ran=\det\big(\lan a_i,b_j\ran\big)
\end{equation}
Fixing a simple unit $n$-vector $A=\ep_1\w \cdots\w \ep_n$, we define
the
$w$-function on $M$:
\begin{equation}\label{wf}
w(p):=\lan e_1\w\cdots\w e_n,A\ran=\det\big(\lan e_i,\ep_j\ran\ran\big).
\end{equation}
Via the Pl\"{u}cker imbedding, the Grassmann manifold $\grs{n}{m}$ can be viewed as a submanifold in a Euclidean space,
and the $w$-function can be regarded as the composition of the Gauss map and a height function on
$\grs{n}{m}$ (see \cite{x-y1}, \cite{j-x-y2} and \cite{j-x-y3}). In particular, if $M=\big(x,f(x)\big)$ is a graph
in $\R^{n+m}$ given by a vector-valued function $f:\R^n\ra \R^m$, then choosing $A$ to be one representing
$(x_1,\cdots,x_n)$ coordinate $n$-plane implies $w>0$ and moreover
\begin{equation}
v:=w^{-1}
\end{equation}
equals the volume element of $M$ (see \cite{j-x}).

The Codazzi equations yield the following formulas for the $w$-function:

\begin{lem}\cite{fc}\cite{x1}
If $M$ is a submanifold in $\R^{n+m}$, then
\begin{equation}\label{dw}
\n_{e_i} w=h_{\a,ij}\lan e_{j\a},\ep_1\w\cdots\w \ep_n\ran.
\end{equation}
 Moreover if $M$
has parallel mean curvature, i.e. $\n H\equiv 0$, then
\begin{equation}\label{La}
\De w=-|B|^2 w+\sum_i\sum_{\a\neq \be,j\neq k}h_{\a,ij} h_{\be,ik} \lan e_{j\a,k\be},\ep_1\w \cdots\w \ep_n\ran.
\end{equation}
with
\begin{equation}
e_{j\a,k\be}=e_1\w\cdots\w \nu_\a \w\cdots\w \nu_\be\w \cdots\w e_n
\end{equation}
that is obtained by replacing $e_j$ by $\nu_\a$ and $e_k$ by $\nu_\be$ in $e_1\w \cdots\w e_n$, respectively.
\end{lem}

To have the curvature estimates we need the Simons' version of the Bochner type formula for the squared norm of the second fundamental form. A straightforward calculation shows (see \cite{Si}, (2.6) in \cite{x1})
\begin{equation}\label{b4}
\De |B|^2=2|\n B|^2+2\lan \n^2 B,B\ran\geq 2|\n B|^2-3|B|^4.
\end{equation}

\Section{Small G-rank cases}{Small G-rank cases}

The rank of the Gauss map for a submanifold $M$ in $\ir{n+m}$ is closely related to rigidity problems.
 The classical Beez-Killing theorem is a local rigidity property for
hypersurfaces in $\ir{n+1}$ when $G-rank\ge 3$. For global
investigations,  we refer to \cite{d-g}. Here, we
study the case of $G-rank\le 2$.

Now, for every $p\in M$, we have
$$\dim \text{Ker}(\g_*)_p=n-\text{rank}(\g_*)_p\geq n-2.$$
Then for any $p_0\in M$, there exists a local smooth distribution
$\mc{K}$ of dimension $n-2$ on $U\ni p_0$, such that
$\mc{K}_p\subset \text{Ker}(\g_*)_p$ for any $p\in U$. $\mc{K}$ is
called the relative nullity distribution by Chern-Kuiper \cite{c-k}.
This is an integrable distribution. Therefore, one can find a local
tangent orthonormal frame field $\{e_i\}$, such that
$\mc{K}_p=\text{span}\{e_i(p):i\geq 3\}$, i.e.
\begin{equation}
\g_*e_i=0\qquad 3\leq i\leq n
\end{equation}
and it  follows from (\ref{Gm1}) that $h_{\a,ij}=0$, i.e.
\begin{equation}
B_{e_i e_j}=0\qquad \text{whenever }i\geq 3\text{ or }j\geq 3.
\end{equation}
Hence
\begin{equation}
0=H=\sum_{i=1}^n B_{e_i e_i}=B_{e_1 e_1}+B_{e_2 e_2}.
\end{equation}
At the considered point, let
\begin{equation}
G(e_1,e_2):=\left(\begin{array}{cc}
\lan B_{e_1 e_1},B_{e_1 e_1}\ran & \lan B_{e_1 e_1},B_{e_1 e_2}\ran\\
\lan B_{e_1 e_2},B_{e_1,e_1}\ran & \lan B_{e_1 e_2},B_{e_1 e_2}\ran
\end{array}\right).
\end{equation}
$G$ then  is a semi-positive definite matrix, whose eigenvalues are
denoted by $\mu_1^2$ and $\mu_2^2$ ($\mu_1\geq \mu_2\geq 0$) . Then there exists
an orthogonal matrix
\begin{equation}\label{o1}
O=\left(\begin{array}{cc}
\cos\th & -\sin\th\\
\sin\th & \cos\th
\end{array}\right)
\end{equation}
such that
\begin{equation}\label{o2}
G=O\left(\begin{array}{cc}
\mu_1^2 & \\
  & \mu_2^2
\end{array}\right)O^T.
\end{equation}
Now we put
\begin{equation}
f_1=\cos\a e_1-\sin\a e_2\qquad f_2=\sin\a e_1+\cos\a e_2
\end{equation}
with $\a$ to be chosen,
then
$$\aligned
B_{f_1f_1}&=\cos^2\a B_{e_1e_1}+\sin^2\a B_{e_2e_2}-2\cos\a\sin\a B_{e_1e_2}\\
&=\cos(2\a) B_{e_1e_1}-\sin(2\a) B_{e_1e_2}\\
B_{f_1f_2}&=\cos\a\sin\a B_{e_1e_1}-\cos\a\sin\a B_{e_2e_2}+(\cos^2\a-\sin^2\a)B_{e_1e_2}\\
&=\sin(2\a) B_{e_1e_1}+\cos(2\a) B_{e_1e_2}.
\endaligned$$
Thus
\begin{equation}\label{o3}
G(f_1,f_2)=\left(\begin{array}{cc}
\cos(2\a) & -\sin(2\a)\\
\sin(2\a) & \cos(2\a)
\end{array}\right)G(e_1,e_2)\left(\begin{array}{cc}
\cos(2\a) & -\sin(2\a)\\
\sin(2\a) & \cos(2\a)
\end{array}\right)^T
\end{equation}
Choosing $\a=-\f{\th}{2}$ and combining with (\ref{o1}), (\ref{o2}) and (\ref{o3}) gives
$$G(f_1,f_2)=\left(\begin{array}{cc}
\mu_1^2 & \\
 & \mu_2^2
\end{array}\right).
$$
Therefore, by carefully choosing local tangent frames and normal frames, one can assume that at the considered point
\begin{equation}\label{sf}
A^1=\left(\begin{array}{cccc}
\mu_1 & 0 & &\\
0 & -\mu_1 & &\\
& & &\\
& & \multicolumn{1}{c}{\raisebox{0.5ex}[0pt]{\Huge o}}&
\end{array}
\right)\qquad
A^2=\left(\begin{array}{cccc}
0 & \mu_2 & &\\
\mu_2 & 0 & &\\
& & &\\
& & \multicolumn{1}{c}{\raisebox{0.5ex}[0pt]{\Huge o}}&
\end{array}\right)
\end{equation}
and $A^\a=0$ for each $\a\geq 3$, where $A^\a:=A^{\nu_\a}$ is the shape operator.

In this case, (\ref{La}) can be rewritten as
\begin{equation}\label{La2}\aligned
\De w&=-|B|^2 w+2\sum_i\sum_{j\neq k}h_{1,ij}h_{2,ik}\lan e_{j1,k2},A\ran\\
&=-|B|^2 w+2h_{1,11}h_{2,12}\lan e_{11,22},A\ran+2h_{1,22}h_{2,21}\lan e_{21,12},A\ran\\
&=-|B|^2 w+4\mu_1\mu_2\lan e_{11,22},A\ran
\endaligned
\end{equation}
where $A=\ep_1\w\cdots\w\ep_n$ and the last step follows from $e_{11,22}=-e_{21,12}=\nu_1\w\nu_2\w e_3\w\cdots\w e_n$.
By (\ref{dw}),
$$\aligned
\n_{e_1}w&=h_{1,11}\lan e_{11},A\ran+h_{2,12}\lan e_{22},A\ran=\mu_1\lan e_{11},A\ran+\mu_2\lan e_{22},A\ran\\
\n_{e_2}w&=h_{1,22}\lan e_{21},A\ran+h_{2,21}\lan e_{12},A\ran=-\mu_1\lan e_{21},A\ran+\mu_2\lan e_{12},A\ran
\endaligned$$
and $\n_{e_i}w=0$ for every $i\geq 3$. Hence
\begin{equation}\label{dw2}
\aligned
|\n w|^2=&\sum_i |\n_{e_i}w|^2=\big(\mu_1\lan e_{11},A\ran+\mu_2\lan e_{22},A\ran\big)^2+\big(-\mu_1\lan e_{21},A\ran+\mu_2\lan e_{12},A\ran\big)^2\\
=&\big(\mu_1\lan e_{11},A\ran-\mu_2\lan e_{22},A\ran\big)^2+\big( \mu_1\lan e_{21},A\ran+\mu_2\lan e_{12},A\ran\big)^2\\
&+4\mu_1\mu_2\big(\lan e_{11},A\ran\lan e_{22},A\ran-\lan e_{21},A\ran\lan e_{12},A\ran\big)
\endaligned
\end{equation}
By Lemma 3.2 of \cite{j-x-y3},
\begin{equation}\label{eq1}
\lan e_1\w\cdots\w e_n,A\ran\lan e_{11,22},A\ran-\lan e_{11},A\ran\lan e_{22},A\ran+\lan e_{12},A\ran\lan e_{21},A\ran=0.
\end{equation}
In conjunction with (\ref{La2}), (\ref{dw2}) and (\ref{eq1}), we have
\begin{equation}\aligned
\De \log w&=w^{-2}(w\De w-|\n w|^2)\\
&=-|B|^2-w^{-2}\Big[\big(\mu_1\lan e_{11},A\ran-\mu_2\lan e_{22},A\ran\big)^2+\big( \mu_1\lan e_{21},A\ran+\mu_2\lan e_{12},A\ran\big)^2\Big]
\endaligned
\end{equation}
whenever $w>0$. We thus have the following results from  our previous paper

\begin{pro}\cite{j-x-y3}
Let $M$ be a minimal submanifold of $\R^{n+m}$ with $G-rank\le 2$ and $w>0$. Then,
\begin{equation}\label{dew}
\De \log w\leq -|B|^2.
\end{equation}
\end{pro}

\begin{defi}\label{G-conformall}
Let $M$ be an $n$-dimensional minimal  submanifold in $\ir{n+m}$ , ($m\geq 2$).  A point $p\in M$ is called a \textit{G-conformal point}, if there exists an orthonormal basis $\{e_1,\cdots,e_n\}$ of $T_p M$,
such that $B_{e_i e_j}=0$ whenever $i\geq 3$ or $j\geq 3$, and
$$\lan B_{e_1 e_1},B_{e_1 e_2}\ran=0,\qquad |B_{e_1 e_1}|=|B_{e_1 e_2}|.$$
Moreover if each point of $M$ is a G-conformal point, we call $M$  a \textit{totally G-conformal minimal submanifold}.
\end{defi}

The formula (\ref{b4}) is for minimal submanifolds in $\ir{n+m}$ with codimension $m\ge 2$. In the present situation we can derive it directly and
 analyze its accuracy.  A straightforward calculation shows (see \cite{Si} \cite{x})
\begin{equation}\label{b}
\n^2 B=-\td{\mc{B}}-\ul{\mc{B}}.
\end{equation}
Here $\n^2$ denotes the trace-Laplace operator acting on any cross-section
of a vector bundle over $M$,
\begin{equation}
\td{\mc{B}}:=B\circ B^t\circ B
\end{equation}
with $B^t$ denoting the conjugate map of $B$, and
\begin{equation}
\ul{\mc{B}}:=\sum_{\a=1}^m \big(B_{A^\a A^\a(X),Y}+B_{X,A^\a A^\a(Y)}-2B_{A^\a(X),A^\a(Y)}\big).
\end{equation}
Hence
\begin{equation}\label{b1}
\aligned
\lan \td{\mc{B}},B\ran&=\lan B\circ B^t\circ B,B\ran=\lan B^t\circ B,B^t\circ B\ran\\
&=\lan B_{e_ie_j},B_{e_ke_l}\ran\lan B_{e_i e_j},B_{e_k e_l}\ran=h_{\a,ij}h_{\a,kl}h_{\be,ij}h_{\be,kl}\\
&=h_{\a,ij}h_{\be,ji}h_{\a,kl}h_{\be,lk}=(A^\a A^\be)_{ii}(A^\a A^\be)_{kk}\\
&=\sum_{\a,\be}\big[\tr(A^\a A^\be)\big]^2=4\mu_1^4+4\mu_2^4
\endaligned
\end{equation}
where the last step follows from (\ref{sf}), and
\begin{equation}\label{b2}
\aligned
\lan \ul{\mc{B}},B\ran=&\lan B_{A^\a A^\a(e_i),e_j}+B_{e_i,A^\a A^\a(e_j)}-2B_{A^\a(e_i),A^\a(e_j)},\nu_\be\ran \lan B_{e_i,e_j},\nu_\be\ran\\
=&\lan A^\be A^\a A^\a(e_i),e_j\ran \lan A^\be(e_j),e_i\ran+\lan A^\be A^\a A^\a(e_j),e_i\ran\lan A^\be(e_i),e_j\ran\\
 &-2\lan A^\be A^\a(e_i),A^\a(e_j)\ran \lan A^\be(e_j),e_i\ran\\
=&(A^\be A^\a A^\a)_{ij}(A^\be)_{ji}+(A^\be A^\a A^\a)_{ji}(A^\be)_{ij}-2(A^\a A^\be A^\a)_{ij}(A^\be)_{ji}\\
=&2\tr(A^\be A^\a A^\a A^\be-A^\a A^\be A^\a A^\be)=2\tr\big([A^\be,A^\a]A^\a A^\be\big)\\
=&\tr\big([A^\be,A^\a]A^\a A^\be\big)+\tr\big([A^\a,A^\be]A^\be A^\a\big)=-\sum_{\a,\be}\tr\big([A^\a,A^\be]^2\big)\\
=&-2\tr\big([A^1,A^2]^2\big)=16\mu_1^2\mu_2^2
\endaligned
\end{equation}
where
\begin{equation}
[A^1,A^2]=\left(\begin{array}{cccc}
0 & 2\mu_1\mu_2 & &\\
-2\mu_1\mu_2 & 0 & &\\
& & &\\
& & \multicolumn{1}{c}{\raisebox{0.5ex}[0pt]{\Huge o}}&
\end{array}\right).
\end{equation}
Substituting (\ref{b1}) and (\ref{b2}) into (\ref{b}) gives
\begin{equation}\label{b3}
\aligned
-\f{\lan \n^2 B,B\ran}{|B|^4}&=\f{\lan \td{\mc{B}}+\ul{\mc{B}},B\ran}{|B|^4}
=\f{4\mu_1^4+4\mu_2^4+16\mu_1^2\mu_2^2}{(2\mu_1^2+2\mu_2^2)^2}\\
&=1+\f{2\mu_1^2\mu_2^2}{(\mu_1^2+\mu_2^2)^2}\leq \f{3}{2}
\endaligned
\end{equation}
where the equality holds if and only if $\mu_1=\mu_2$.

\begin{pro}\label{simons}
Let $M$ be an $n$-dimensional minimal submanifold in $\ir{n+m}$ with codimension
$m\geq 2$. Then
\begin{equation*}
\De |B|^2\geq 2|\n B|^2-3|B|^4.
\end{equation*}
In the case of $G-rank\le 2$  the equality holds at $p\in M$ if and only if $p$ is a G-conformal point.

\end{pro}

In order to make use of the formula (\ref{b4}), we also need to
estimate $|\n B|^2$ in terms of $|\n |B||^2$. Schoen-Simon-Yau
\cite{s-s-y} obtained such an estimate for the hypersurface case. It
was generalized  to arbitrary codimension in \cite{x-y1} and refined
and generalized in \cite{x2}. In particular, if  $G-rank\le 2$ for $M$, we have a more precise estimate.

\begin{pro}\label{kato1}
If $M$ is an $n$-dimensional minimal  submanifold in $\R^{n+m}$ with $G-rank\le 2$, then
\begin{equation}\label{kato}
|\n B|^2\geq 2\big|\n |B|\big|^2.
\end{equation}
The equality holds at $p\in M$, if and only if there exist an orthonormal basis $\{e_1,\cdots,e_n\}$ of $T_p M$ and $\la_1,\la_2\in \R$, such that
$B_{e_ie_j}=0$ whenever $i\geq 3$ or $j\geq 3$, $\lan B_{e_1 e_1},B_{e_1 e_2}\ran=0$, $(\n_{e_k}B)_{e_i e_j}=0$ whenever $i\geq 3$, $j\geq 3$ or $k\geq 3$, and
\begin{equation}\label{kato2}
 \aligned
(\n_{e_1}B)_{e_1 e_1}&=\la_1 B_{e_1 e_1}-\la_2 B_{e_1 e_2},\\
(\n_{e_2}B)_{e_1 e_1}&=\la_2 B_{e_1 e_1}+\la_1 B_{e_1 e_2}.
\endaligned
\end{equation}
In particular, if $n=2$ and $m=1$, $|\n B|^2= 2\big|\n |B|\big|^2$ holds everywhere.
\end{pro}

\begin{proof}
It is sufficient for us to prove the inequality at the points where $|B|^2\neq 0$.

With the same notation $A^\a,\mu_1,\mu_2$  as in (\ref{sf}), the
triangle inequality yields
\begin{equation}\label{tri}\big|\n |B|^2\big|=\Big|\sum_\a \n |A^\a|^2\Big|\leq \sum_\a \big|\n |A^\a|^2\big|.
\end{equation}
By the Schwarz inequality, we obtain
\begin{equation}\aligned\label{na1}
  \big|\n |B|\big|^2&=\f{\big|\n |B|^2\big|^2}{4|B|^2}\leq\f{\Big(\sum_\a \big|\n |A^\a|^2\big|\Big)^2}{4\sum_\a |A^\a|^2}
= \f{\Big(\sum_\a \f{\nA{\a}}{|A^\a|}\cdot |A^\a|\Big)^2}{4\sum_\a |A^\a|^2}\\
&\leq \f{\sum\Big(\f{{\nA{\a}}^2}{|A^\a|^2}\Big)\cdot \sum_\a |A^\a|^2}{4\sum_\a |A^\a|^2}
=\sum_\a \f{{\nA{\a}}^2}{4 |A^\a|^2}\\
&=\f{{\nA{1}}^2}{4 |A^1|^2}+\f{{\nA{2}}^2}{4 |A^2|^2}.
\endaligned
\end{equation}
Note that here and in the sequel we set
$ \f{{\nA{\a}}^2}{4 |A^\a|^2}=0$ whenever $|A^\a|=0$.

Since $|A^\a|^2=\sum_{i,j}h_{\a,ij}^2$,
\begin{equation}\label{na}
\n_{e_k}|A^\a|^2=2h_{\a,ij}h_{\a,ijk}
\end{equation}
with
\begin{equation}
h_{\a,ijk}:=\lan (\n_{e_k}B)_{e_i e_j},\nu_\a\ran.
\end{equation}
As shown above, the assumption  $G-rank\le 2$ implies the existence
of a local orthonormal tangent frame field $\{e_i\}$ on an open
domain $U$ as shown before, such that $B_{e_i e_j}\equiv 0$ whenever
$i\geq 3$ or $j\geq 3$. Hence for arbitrary $i,j\geq 3$,
$$0=\n_{e_k}(B_{e_i e_j})=(\n_{e_k}B)_{e_i e_j}+B_{\n_{e_k}e_i,e_j}+B_{e_i,\n_{e_k}e_j}=(\n_{e_k}B)_{e_i e_j}$$
holds for all $k$, i.e.
\begin{equation}\label{b6}
h_{\a,ijk}=0\qquad \forall i,j\geq 3.
\end{equation}
It immediately follows that
\begin{equation}\label{b5}
0=\lan \n_{e_k} H,\nu_\a\ran=\sum_i h_{\a,iik}=h_{\a,11k}+h_{\a,22k}.
\end{equation}
In conjunction with (\ref{sf}), (\ref{na}) and (\ref{b5}), we get
$$\aligned
\nA{1}^2&=4\sum_k\big(\sum_{i,j}h_{1,ij}h_{1,ijk}\big)^2=4\sum_k(h_{1,11}h_{1,11k}+h_{1,22}h_{1,22k})^2\\
&=16\mu_1^2\sum_k h_{1,11k}^2=8|A^1|^2\sum_k h_{1,11k}^2
\endaligned$$
and moreover
\begin{equation}\label{na2}
\f{\nA{1}^2}{|A^1|^2}=8\sum_k h_{1,11k}^2.
\end{equation}
A similar calculation shows
\begin{equation}\label{na3}
\f{\nA{2}^2}{|A^2|^2}=8\sum_k h_{2,12k}^2.
\end{equation}
Substituting (\ref{na2}) and (\ref{na3}) into (\ref{na1}) implies
\begin{equation}\label{na4}
\big|\n|B|\big|^2\leq 2\sum_{k}h_{1,11k}^2+2\sum_k h_{2,12k}^2.
\end{equation}

On the other hand,
\begin{equation}\aligned\label{na5}
|\n B|^2=&\sum_{\a,i,j,k}h_{\a,ijk}^2\geq \sum_{i,j,k}h_{1,ijk}^2+\sum_{i,j,k}h_{2,ijk}^2\\
=&\ (h_{1,111}^2+h_{1,221}^2+h_{1,122}^2+h_{1,212}^2)+(h_{1,112}^2+h_{1,121}^2+h_{1,211}^2+h_{1,222}^2)\\
&+\sum_{k\geq 3}(h_{1,11k}^2+h_{1,1k1}^2+h_{1,k11}^2+h_{1,22k}^2+h_{1,2k2}^2+h_{1,k22}^2)\\
&+(h_{2,121}^2+h_{2,112}^2+h_{2,211}^2+h_{2,222}^2)+(h_{2,122}^2+h_{2,212}^2+h_{2,221}^2+h_{2,111}^2)\\
&+\sum_{k\geq 3}(h_{2,12k}^2+h_{2,k12}^2+h_{2,2k1}^2+h_{2,21k}^2+h_{2,k21}^2+h_{2,1k2}^2)\\
\geq& 4\sum_k h_{1,11k}^2+4\sum_k h_{2,12k}^2.
\endaligned
\end{equation}
Here we have used (\ref{b6}), (\ref{b5}) and $h_{\a,ijk}=h_{\a,ikj}$,
which is an immediate corollary of the Codazzi equations. Combining
this with (\ref{na4}) and (\ref{na5}) yields (\ref{kato}).

Now we determine the conditions ensuring that equality in (\ref{kato}) holds true at $p\in M$. Obviously $|\n B|^2=2\big|\n |B|\big|^2$ requires all the
equalities in (\ref{tri}), (\ref{na1}) and (\ref{na5}) hold simultaneously.

It is easily seen that  equality holds in (\ref{na5}) if and only if $h_{\a,ijk}=0$ whenever one of the indices is no less than $3$. Hence by (\ref{na}),
\begin{equation}\aligned
 \n |A^1|^2=&(2h_{1,11}h_{1,111}+2h_{1,22}h_{1,221})e_1+(2h_{1,11}h_{1,112}+2h_{1,22}h_{1,222})e_2\\
 =&4\mu_1(h_{1,111}e_1+h_{1,112}e_2),\\
\n |A^2|^2=&(2h_{2,12}h_{2,121}+2h_{2,21}h_{2,211})e_1+(2h_{2,12}h_{2,122}+2h_{2,21}h_{2,212})e_2\\
=&4\mu_2(h_{2,112}e_1-h_{2,111}e_2),
\endaligned
\end{equation}
and
\begin{equation}\aligned
v_1:=\f{\n |A^1|^2}{|A^1|}&=2\sqrt{2}(h_{1,111}e_1+h_{1,112}e_2),\\
v_2:=\f{\n |A^2|^2}{|A^2|}&=2\sqrt{2}(h_{2,112}e_1-h_{2,111}e_2).
\endaligned
\end{equation}
(\ref{tri}) and (\ref{na1}) hold true if and only if the following 2 conditions are satisfied: (i) $\n |A^1|^2$ and $\n |A^2|^2$ point in the same direction, (ii)
 $(|v_1|,|v_2|)$ and $(\mu_1,\mu_2)$ are linearly depedent. Hence there exist $\la_1,\la_2\in \R$, such that
$$\aligned
h_{1,111}e_1+h_{1,112}e_2&=\mu_1(\la_1 e_1+\la_2 e_2),\\
h_{2,112}e_1-h_{2,111}e_2&=\mu_2(\la_1 e_1+\la_2 e_2).
\endaligned$$
This is equivalent to
\begin{equation}\aligned
 (\n_{e_1}B)_{e_1 e_1}&=\mu_1\la_1 \nu_1-\mu_2\la_2 \nu_2=\la_1 B_{e_1 e_1}-\la_2 B_{e_1 e_2},\\
 (\n_{e_2}B)_{e_1 e_1}&=\mu_1\la_2 \nu_1+\mu_2\la_1 \nu_2=\la_2 B_{e_1 e_1}+\la_1 B_{e_1 e_2}.
\endaligned
\end{equation}
\end{proof}

In conjunction with (\ref{b4}) and (\ref{kato}), we arrive at
\begin{equation}\label{si2}
\De|B|^2\geq 4\big|\n|B|\big|^2-3|B|^4.
\end{equation}

\Section{Curvature estimates}{Curvature estimates}

We are ready to derive the curvature estimates, in a manner similar to \cite{e-h}. When $w>0$, we put $v:=w^{-1}$, then (\ref{dew}) is equivalent to
\begin{equation}\label{dew20}
\De v\geq |B|^2 v+v^{-1}|\n v|^2.
\end{equation}
From (\ref{dew20}) and (\ref{si2}), a straightforward calculation shows
$$\aligned
&\De\big(|B|^{2s}v^q\big)\\
=&\De\big(|B|^{2s}\big)v^q+|B|^{2s}\De v^q+2\lan \n |B|^{2s},\n v^q\ran\\
\geq&s|B|^{2s-2}\Big(4\big|\n |B|\big|^2-3|B|^4\Big)v^q+4s(s-1)|B|^{2s-2}\big|\n |B|\big|^2 v^q\\
&+q|B|^{2s}v^{q-1}\big(|B|^2 v+v^{-1}|\n v|^2\big)+q(q-1)|B|^{2s}v^{q-2}|\n v|^2\\
&+4sq|B|^{2s-1}v^{q-1}\lan \n|B|,\n v\ran\\
\geq&(-3s+q)|B|^{2s+2}v^q+4s^2|B|^{2s-2}\big|\n |B|\big|^2v^q\\
&+q^2|B|^{2s}v^{q-2}|\n v|^2+4sq|B|^{2s-1}v^{q-1}\lan \n |B|,\n v\ran.
\endaligned$$
It follows that
\begin{equation}\label{es3}
\De\big(|B|^{2s}v^q\big)\geq (-3s+q)|B|^{2s+2}v^q
\end{equation}
for arbitrary $s,q\geq 1$.

Let $t=2s+1$, then
\begin{equation}
\De\big(|B|^{t-1}v^q\big)\geq \big(q-\f{3t-3}{2}\big)|B|^{t+1}v^q
\end{equation}
for arbitrary $t\geq 3$ and $q\geq 1$. Whenever $q>\f{3t-3}{2}$,
putting $C_1(t,q)=\big(q-\f{3t-3}{2}\big)^{-1}$ gives
\begin{equation}
|B|^{2t}v^{2q}\eta^{2t}\leq C_1 \De\big(|B|^{t-1}v^q\big)|B|^{t-1}v^q\eta^{2t}
\end{equation}
with $\eta$ being an arbitrary smooth function in $M$ with compact supporting set. Integrating both sides of the above inequality over $M$ implies
$$\aligned
&\int_M |B|^{2t}v^{2q}\eta^{2t}*1\\
\leq &C_1\int_M \De\big(|B|^{t-1}v^q\big)|B|^{t-1}v^q\eta^{2t}*1\\
=&-C_1\int_M \Big\lan \n\big(|B|^{t-1}v^q\big),\n\big(|B|^{t-1}v^q \eta^{2t}\big)\Big\ran*1\\
=&-C_1\int_M \Big|\n\big(|B|^{t-1}v^q\big)\Big|^2\eta^{2t}*1-2tC_1\int_M |B|^{t-1}v^q\eta^{2t-1}  \lan \n\big(|B|^{t-1}v^q\big),\n \eta\ran*1\\
\leq&-C_1\int_M \Big|\n\big(|B|^{t-1}v^q\big)\Big|^2\eta^{2t}*1+C_1\int_M \Big|\n\big(|B|^{t-1}v^q\big)\Big|^2\eta^{2t}*1\\
&+C_1t^2\int_M |B|^{2t-2}v^{2q}\eta^{2t-2}|\n \eta|^2*1\\
\leq&C_1t^2\Big(\f{t-1}{t}\ep^{\f{t}{t-1}}\int_M |B|^{2t}v^{2q}\eta^{2t}*1+\f{1}{t}\ep^{-t}\int_M v^{2q}|\n \eta|^{2t}*1\Big)
\endaligned
$$
for arbitrary $\ep>0$. Here we have used
Stokes' theorem and Young's inequality.
Choosing $\ep$ such that $C_1t(t-1)\ep^{\f{t}{t-1}}=\f{1}{2}$ gives
\begin{equation}\label{es1}
\Big(\int_M |B|^{2t}v^{2q}\eta^{2t}*1\Big)^{\f{1}{t}}\leq C_2(t,q)\Big(\int_M v^{2q}|\n \eta|^{2t}*1\Big)^{\f{1}{t}}
\end{equation}
for arbitrary $t\geq 3$ and $q>\f{3t-3}{2}$.

\begin{thm}\label{t1}
Let $M$ be an $n$-dimensional minimal submanifold (not necessarily complete)  in $\R^{n+m}$ with $G-rank\le 2$ and positive $w$-function  on $M$. Let $\rho:M\times M\ra \R$ be a distance function on $M$, such that $|\n \rho(\cdot,p)|\leq 1$ for each $p\in M$.
 Fix $p_0\in M$,  and denote by $B_R=B_R(p_0):=\{p\in M:\rho(p,p_0)<R\}$ the distance ball centered at $p_0$ and of radius $R$.
Assume $B_{R_0}\subset B_R\subset\subset M$, then for arbitrary $t\geq 3$ and $q>\f{3t-3}{2}$, there exists a positive constant $C_3$, depending only on $t$ and $q$, such that
\begin{equation}\label{es2}
\big\| |B|^{2}v^{\f{2q}{t}}\big\|_{L^t(B_{R_0})} \leq C_3(R-R_0)^{-2}\big\|v^{\f{2q}{t}}\big\|_{L^t(B_R)}.
\end{equation}
with $v:=w^{-1}$.
\end{thm}

\begin{proof} We let
 $\psi$  be a standard bump function on $[0,\infty)$ with $\text{supp}(\psi)\subset [0,R)$, $\psi\equiv 1$ on $[0,R_0]$ and $|\psi'|\leq c_0(R-R_0)^{-1}$. Inserting $\eta=\psi\circ \rho(\cdot,p_0)$ in (\ref{es1}), we have
\begin{equation}
\aligned
&\big\| |B|^2v^{\f{2q}{t}}\big\|_{L^t(B_{R_0})}=\Big(\int_{B_{R_0}}|B|^{2t}v^{2q}*1\Big)^{\f{1}{t}}\leq
\Big(\int_M |B|^{2t}v^{2q}\eta^{2t}*1\Big)^{\f{1}{t}}\\
\leq& C_2\Big(\int_M v^{2q}|\n \eta|^{2t}*1\Big)^{\f{1}{t}}
=C_2\Big(\int_{B_R}v^{2q}|\psi'|^{2t}|\n \rho(\cdot,p_0)|^{2t}*1\Big)^{\f{1}{t}}\\
\leq& C_3(R-R_0)^{-2}\Big(\int_{B_R} v^{2q}*1\Big)^{\f{1}{t}}
=C_3(R-R_0)^{-2}\big\|v^{\f{2q}{t}}\big\|_{L^t(B_R)}.
\endaligned
\end{equation}
\end{proof}

Furthermore, the mean value inequality for subharmonic functions on minimal
submanifolds in Euclidean space  can be applied to deduce a pointwise
estimate for $|B|^2$.

\begin{thm}\label{t2}
Our assumption of $M$ is the same as in Theorem \ref{t1}. Denote by $D_R=D_R(p_0)$ the exterior ball centered at $p_0$
and of radius $R$, then for every $t\geq 3$, there exists a positive constant $C_4$ only depending on $t$, such that
\begin{equation}\label{es4}
(|B|^2 v^3)(p_0)\leq C_4 R^{-2}(\max_{D_R}v)^3\Big(\f{V(R)}{V(\f{R}{2})}\Big)^{\f{1}{t}}.
\end{equation}
Here $V(R)=V(p_0,R):=\text{Vol}(D_R(p_0)).$
\end{thm}

\begin{proof}
 Let $F: M\ra \R^{n+m}$ be the isomorphic immersion and denote by
 $r:M\times M\ra \R$ the restriction of the Euclidean distance function. Without loss of generality one can assume $F(p_0)=0$ for $p_0\in M$, then
$r^2(\cdot,p_0)=\lan F,F\ran.$ This extrinsic distance function $r$ on $M$ satisfies the assumptions of Theorem \ref{t1}.

Letting $q=\f{3t}{2}$ in (\ref{es1}) yields
\begin{equation}\label{es5}
\Big(\int_M |B|^{2t}v^{3t}\eta^{2t}*1\Big)^{\f{1}{t}}\leq C_2\Big(\int_M v^{3t}|\n \eta|^{2t}*1\Big)^{\f{1}{t}}.
\end{equation}
Let $\eta$  be a cut-off function on $M$ with $\text{supp}\
\eta\subset B_R$, $\eta|_{B_{\f{R}{2}}}\equiv 1$ and $|\n \eta|\leq
c_0R^{-1}$ (the construction of the auxiliary function is the same as in Theorem \ref{t1}). Then
\begin{equation}
\Big(\int_M v^{3t}|\n \eta|^{2t}*1\Big)^{\f{1}{t}}\leq C_5(t)R^{-2}(\max_{D_R}v)^3V(R)^{\f{1}{t}}.
\end{equation}
By (\ref{es3}), $|B|^{2t}v^{3t}$ is a subharmonic function on $M$, and
by the mean value inequality,
\begin{equation}\label{es6}
\Big(\int_M |B|^{2t}v^{3t}\eta^{2t}*1\Big)^{\f{1}{t}}\geq \Big(\int_{D_{\f{R}{2}}}|B|^{2t}v^{3t}*1\Big)^{\f{1}{t}}\geq (|B|^2 v^3)(p_0)V\left(\f{R}{2}\right)^{\f{1}{t}}.
\end{equation}
In conjunction with (\ref{es5})-(\ref{es6}) we arrive at (\ref{es4}).

\end{proof}

From the preceding curvature estimates we immediately get the following Bernstein type theorem.

\begin{thm}
Let $M$ be an $n$-dimensional complete minimal submanifold in
$\R^{n+m}$ with $G-rank\le 2$ and a positive $w$-function.
If $M$ has polynomial volume growth and the function $v=w^{-1}$  has growth
\begin{equation}
\max_{D_R(p_0)}v=o(R^{\f{2}{3}})
\end{equation}
for a fixed point $p_0$, then $M$ has to be an affine linear subspace.
\end{thm}

\begin{rem}
Here, we say that
$M$ has polynomial volume growth iff there exists $l\geq 0$ with $V(R)=V(p_0,R)=O(R^l)$.
\end{rem}

\begin{proof}

Let $c_1$ be a positive constant such that
\begin{equation}\label{growth}
V(R)\leq c_1 R^l.
\end{equation}
Now we claim
\begin{equation}\label{claim}
\liminf_{k\ra \infty}\f{V(2^{k+1})}{V(2^k)}\leq 2^l.
\end{equation}
Otherwise, there are $\ep>0$ and a positive integer $N$, such that for any
$k\geq N$,
$$\f{V(2^{k+1})}{V(2^k)}\geq 2^l+\ep.$$
Thus,
$$\f{V(2^k)}{(2^k)^l}\geq \f{V(2^N)(2^l+\ep)^{k-N}}{(2^N)^l(2^l)^{k-N}}
=\f{V(2^N)}{(2^N)^l}\Big(\f{2^l+\ep}{2^l}\Big)^{k-N}.$$
It follows that
$$\lim_{k\ra \infty} \f{V(2^k)}{(2^k)^l}=+\infty$$
which contradicts  (\ref{growth}).

(\ref{claim}) implies the existence of a sequence $\{k_i:i\in \Bbb{N}\}$,
such that $k_i<k_j$ whenever $i<j$, $\lim_{i\ra \infty}k_i=\infty$ and
$$\f{V(2^{k_i+1})}{V(2^{k_i})}\leq 2^l.$$
 then putting $R=R_i:=2^{k_i+1}$ and letting $t=3$ in (\ref{es4}) give
\begin{equation}
(|B|^2 v^3)(p_0)\leq C_4 2^{\f{l}{3}}R_i^{-2}(\max_{D_{R_i}}v)^3
\end{equation}
 Since $\max_{D_R}v=o(R^{\f{2}{3}})$, letting $i\ra \infty$ yields $|B|^2=0$ at $p_0$.

For arbitrary $p\in M$, put $R_0:=r(p,p_0)$, then
the triangle inequality implies $D_{R}(p)\subset D_{R+R_0}(p_0)$ for any $R\geq 0$,
hence
$$\f{V(p,R)}{R^l}\leq \f{V(p_0,R+R_0)}{R^l}\leq \f{c_1(R+R_0)^l}{R^l}$$
which means $V(p,R)=O(R^l)$. Similarly one can show $\max_{D_R(p)}v=o(R^{\f{2}{3}})$
for arbitrary $p$. Thereby one can proceed as above to arrive at $|B|^2=0$ at $p$. Hence $|B|\equiv 0$ on $M$ and $M$ has to be affine linear.

\end{proof}

\Section{Graphical cases}{Graphical cases}

Let $f=(f^1,\cdots,f^n):\Om\subset \R^n\ra \R^m$ be a vector-valued function, then the graph
$M=\{(x,f(x)):x\in \Om\}$ is an embedded  submanifold in $\R^{n+m}$. Let $\{\ep_i,\ep_{n+\a}\}$ be the standard orthonormal basis, and put $A=\ep_1\w \cdots\w \ep_n$,
then as shown in \cite{j-x}, the $w$-function is positive everywhere on $M$ and the volume element of $M$ is
\begin{equation}
*1=v\ dx^1\w\cdots\w dx^n,
\end{equation}
where
\begin{equation}\label{v-f}
v=w^{-1}=\left[\det\Big(\de_{ij}+\sum_\a \f{\p f^\a}{\p x^i}\f{\p f^\a}{\p x^j}\Big)\right]^{\f{1}{2}}.
\end{equation}

Without loss of generality we can assume $f(0)=0$. Denote $p_0=(0,0)$,
then
\begin{equation}
D_R=D_R(p_0)=\{(x,f(x)):|x|^2+|f(x)|^2\leq R^2\}.
\end{equation}
Denote
\begin{equation}
\Om_R=\{x\in \Om:|x|^2+|f(x)|^2\leq R^2\},
\end{equation}
then obviously $\Om_R\subset \Bbb{D}^n(R)$ and $D_R$ is just the graph over $\Om_R$, where $\Bbb{D}^n(R)$ is the $n$-dimensional Euclidean ball of radius $R$.
Hence if
\begin{equation}
\max_{D_R}v\leq CR^l,
\end{equation}
then
\begin{equation}\aligned
V(R)&=\int_{D_R}*1=\int_{\Om_R}v dx^1\w\cdots\w dx^n\\
&\leq \max_{D_R}v\cdot \text{Vol}(\Om_R)\leq CR^l\text{Vol}(\Bbb{D}^n(R))\\
&=C\om_n R^{n+l}
\endaligned
\end{equation}
with $\om_n$ being the volume of the $n$-dimensional unit Euclidean
ball. This  means that
the exterior balls of a graph have polynomial volume growth whenever
the $v$-function has polynomial growth. This fact leads us to  the following result.

\begin{thm}\label{t3}
Let $M=\{(x,f(x)):x\in \R^n\}$ be an entire minimal graph given by a vector-valued function $f:\R^n\ra \R^m$ with  $G-rank\le 2$. If the slope of $f$
satisfies
\begin{equation}
\De_f=\left[\det\Big(\de_{ij}+\sum_{\a}\f{\p f^\a}{\p x^i}\f{\p f^\a}{\p x^j}\Big)\right]^{\f{1}{2}}=o(R^{\f{2}{3}}),
\end{equation}
where $R^2=|x|^2+|f(x)|^2$,
then $f$ has to be an affine linear function.
\end{thm}

Now we study $2$-dimensional cases. It is well-known that every oriented $2$-dimensional Riemannian manifold $M$  admits a local isothermal coordinate chart around any point. More precisely, each $p\in M$ has a coordinate neighborhood
$(U;u,v)$, such that
$$g=\la^2(du^2+dv^2)$$
on $U$ with a positive function $\la$.  In fact, for minimal entire
graphs, one can find a global isothermal coordinate chart:

\begin{lem}(\cite{o} \S 5)\label{iso}
Let $M=\{(x,f(x):x\in \R^2\}$ be a $2$-dimensional entire minimal graph in
$\R^{2+m}$, then there exists a nonsigular linear transformation
\begin{equation}\aligned
u_1&=x_1\\
u_2&=ax_1+bx_2,\qquad (b>0)
\endaligned
\end{equation}
such that $(u_1,u_2)$ are global isothermal parameters for $M$.

\end{lem}

Equipped with this tool, we can obtain another Bernstein type theorem
for entire minimal graphs of dimension $2$.

\begin{thm}\label{t4}
Let $f:\R^2\ra \R^m$ $(x^1,x^2)\mapsto (f^1,\cdots,f^m)$ be an entire
solution of the minimal surface equations
\begin{equation}\label{mi}
\Big(1+\Big|\f{\p f}{\p x^2}\Big|^2\Big)\f{\p^2 f}{(\p x^1)^2}-2\Big\lan \pd{f}{x^1},\pd{f}{x^2}\Big\ran\f{\p^2 f}{\p x^1\p x^2}+\Big(1+\Big|\pd{f}{x^1}\Big|^2\Big)\f{\p^2 f}{(\p x^2)^2}=0.
\end{equation}
If for some $\ep>0$,
\begin{equation}
\De_f=\det\Big(\de_{ij}+\sum_\a \pd{f^\a}{x^i}\pd{f^\a}{x^j}\Big)^{\f{1}{2}}=O(R^{1-\ep})
\end{equation}
with $R=|x|$, then $f$ has to be affine linear.
\end{thm}

\begin{proof}
By Lemma \ref{iso}, one can find a global isothermal coordinate $(u_1,u_2)$
for the  entire minimal graph $M:=\{(x,f(x)):x\in \R^2\}$, i.e.
\begin{equation}\aligned
g&=\la^2\big((du^1)^2+(du^2)^2\big)=\la^2\big((dx^1)^2+(a\ dx^1+b\ dx^2)^2\big)\\
&=\la^2\big((1+a^2)(dx^1)^2+2ab\ dx^1 dx^2+b^2 (dx^2)^2\big).
\endaligned
\end{equation}
In other words, the metric is given by
\begin{equation}\label{metric2}
(g_{ij})=\la^2\left(\begin{array}{cc}
1+a^2 & ab\\
ab & b^2
\end{array}\right).
\end{equation}
Denote the two eigenvalues of $\left(\begin{array}{cc}
1+a^2 & ab\\
ab & b^2
\end{array}\right)$ by $\la_1^2\geq \la_2^2>0$, then
\begin{equation}\label{v1}
v=\det(g_{ij})^{\f{1}{2}}=\la^2\la_1\la_2.
\end{equation}

Since $M$ is a graph, any function $\varphi$ on $M$ can be regarded as a function on $\R^2$. Denote
\begin{equation}
\p_i \varphi=\pd{\varphi}{x^i},\qquad D\varphi=(\p_1 \varphi,\p_2 \varphi)
\end{equation}
and let $\n \varphi$ be the gradient vector of $\varphi$ on $M$ with respect to
$g$. Since the largest eigenvalue of $(g^{ij})$ equals the multiplicative inverse of the smallest eigenvalue of $(g_{ij})$, which is $\la^{-2}\la_2^{-2}$, we have
\begin{equation*}
|\n \varphi|^2=g^{ij}\p_i\varphi\p_j\varphi\leq \la^{-2}\la_2^{-2}|D\varphi|^2
\end{equation*}
i.e.
\begin{equation}
|\n \varphi|\leq \la^{-1}\la_2^{-1}|D\varphi|=\Big(\f{\la_1}{\la_2}\Big)^{\f{1}{2}}v^{-\f{1}{2}}|D\varphi|.
\end{equation}

Given $0<R_0<R$, let $\psi$ be a standard bump function, such that $\text{supp}\ \psi\subset [0,R)$, $\psi\equiv 1$ on $[0,R_0]$ and $|\psi'|\leq c_0(R-R_0)^{-1}$. Taking $\eta(x,f(x))=\psi(|x|)$ in (\ref{es1}) gives
\begin{equation}\aligned
&\Big(\int_{\Bbb{D}^2(R_0)}|B|^{2t}v^{2q+1}dx^1 dx^2\Big)^{\f{1}{t}}\leq
\Big(\int_M |B|^{2t}v^{2q}\eta^{2t}*1\Big)^{\f{1}{t}}\\
\leq&C_2\Big(\int_M v^{2q}|\n\eta|^{2t}*1\Big)^{\f{1}{t}}=C_2\Big(\int_M v^{2q}\Big(\f{\la_1}{\la_2}\Big)^t v^{-t}|D\eta|^{2t}*1\Big)^{\f{1}{t}}\\
\leq&C_6(R-R_0)^{-2}\Big(\int_{\Bbb{D}^2(R)}v^{2q-t+1}dx_1dx_2\Big)^{\f{1}{t}}\\
\leq& C_6(R-R_0)^{-2}(\max_{\Bbb{D}^2(R)} v)^{\f{2q+1}{t}-1}(\pi R^2)^{\f{1}{t}}\\
=&C_7\Big(1-\f{R_0}{R}\Big)^{-2}R^{-2+\f{2}{t}}(\max_{\Bbb{D}^2(R)} v)^{\f{2q+1}{t}-1}
\endaligned
\end{equation}
with $C_6$ and $C_7$ being positive constants depending only on $t,q,a$ and $b$.
Letting $q=\f{3t-1}{2}$ gives $\f{2q+1}{t}-1=2$. Thus the growth condition
of $v$ implies
\begin{equation}
\Big(\int_{\Bbb{D}^2(R_0)}|B|^{2t}v^{3t}dx^1 dx^2\Big)^{\f{1}{t}}\leq C_8\Big(1-\f{R_0}{R}\Big)^{-2}R^{\f{2}{t}-2\ep}.
\end{equation}
Taking $t=\f{2}{\ep}$ and then letting $R\ra +\infty$ force $|B|(x,f(x))=0$
whenever $|x|<R_0$. Finally by letting $R_0\ra +\infty$ we get the Bernstein type result.
\end{proof}

Given a vector-valued function $f:\R^2\ra \R^m$, denote by
$$Df=Df(x):=\Big(\pd{f^\a}{x^i}\Big)$$
the Jacobi matrix of $f$ at $x\in \R^2$. $Df$ can also be seen as a linear mapping
from $\R^2$ to $\R^m$. Obviously $Df (Df)^T$ is a nonnegative definite symmetric matrix,
whose engenvalues are denoted by $\mu_1^2\geq \mu_2^2\geq 0$. It is easy to check that
$\mu_1$ and $\mu_2$ are just the critical values of the function
$$v\in \R^2\backslash 0\mapsto \f{\big|(Df)(v)\big|}{|v|}$$
and for any bounded domain $\mc{D}\subset \R^2$,
$$\mu_1\mu_2=\f{\text{Area}\big(Df(\mc{D})\big)}{\text{Area}(\mc{D})}.$$
In matrix terminology, $\mu_1^2\mu_2^2$ equals the squared sum
of all the $2\times 2$-minors of $Df$, i.e.
\begin{equation}
\mu_1^2\mu_2^2=\sum_{\a<\be}\Big(\pd{f^\a}{x^1}\pd{f^\be}{x^2}-\pd{f^\a}{x^2}\pd{f^\be}{x^1}\Big)^2.
\end{equation}
When $m=2$, $\mu_1\mu_2$ then is the absolute value of  $J_f:=\det(Df)$.

As shown in (\ref{v-f}), the metric matrix of the graph given by $f$
is
\begin{equation}
(g_{ij})=I_2+Df (Df)^T.
\end{equation}
Thus the two eigenvalues of $(g_{ij})$ are $1+\mu_1^2$ and $1+\mu_2^2$, and
\begin{equation}
v^2=\det(g_{ij})=(1+\mu_1^2)(1+\mu_2^2).
\end{equation}

Now we additionally assume that $f$ is an entire solution of the minimal surface equations.
Then as shown in (\ref{metric2}), there exists a positive function $\la$ on $M$ and
two positive constants $\la_1,\la_2$, depending only on $a$ and $b$, such that
\begin{equation}
1+\mu_1^2=\la^2\la_1^2\qquad 1+\mu_2^2=\la^2\la_2^2.
\end{equation}
Hence
\begin{equation}
\aligned
\mu_1^2\mu_2^2&=(\la^2\la_1^2-1)(\la^2\la_2^2-1)=\la_1^2\la_2^2\la^4-(\la_1^2+\la_2^2)\la^2+1\\
&=v^2-\f{\la_1^2+\la_2^2}{\la_1\la_2}v+1.
\endaligned
\end{equation}
Note that $\f{\la_1^2+\la_2^2}{\la_1\la_2}$ is a constant. Once $v$ has polynomial growth,
 $\mu_1\mu_2$ also has polynomial growth of the same order, and vice versa.
Therefore one can obtain an equivalent form of Theorem \ref{t4} as follows.

\begin{thm}
Let $f:\R^2\ra \R^m$ $(x^1,x^2)\mapsto (f^1,\cdots,f^m)$ be an entire
solution of the minimal surface equations.
If for some $\ep>0$,
\begin{equation}\label{Ja1}
\sum_{\a<\be}\Big(\pd{f^\a}{x^1}\pd{f^\be}{x^2}-\pd{f^\a}{x^2}\pd{f^\be}{x^1}\Big)^2=O(R^{2(1-\ep)})
\end{equation}
with $R=|x|$, then $f$ has to be affine linear. If $m=2$, the condition (\ref{Ja1}) is equivalent to
\begin{equation}
|J_f|:=|\det(Df)|=O(R^{1-\ep}).
\end{equation}
\end{thm}

Similarly we have a version of  Theorem \ref{t3} for the minimal surface case.

\begin{thm}
Let $f:\R^2\ra \R^m$ $(x^1,x^2)\mapsto (f^1,\cdots,f^m)$ be an entire
solution of the minimal surface equations.
If
\begin{equation}\label{Ja2}
\sum_{\a<\be}\Big(\pd{f^\a}{x^1}\pd{f^\be}{x^2}-\pd{f^\a}{x^2}\pd{f^\be}{x^1}\Big)^2=o(R^{\f{4}{3}})
\end{equation}
with $R^2=|x|^2+|f(x)|^2$, then $f$ has to be affine linear. If $m=2$, the condition (\ref{Ja2}) is equivalent to
\begin{equation}
|J_f|:=|\det(Df)|=o(R^{\f{2}{3}}).
\end{equation}
\end{thm}

\begin{rem}
Obviously, the above result is also a generalization of that of  \cite{h-s-v1}.
\end{rem}

\Section{Discussions}{Discussions}

We wish to discuss the case of  a minimal surface $M$ in
$\R^{2+m}$. It is natural to ask under which conditions
the equality in  (\ref{b4}), (\ref{dew}) or (\ref{kato}) holds.

With the aid of Lemma \ref{iso}, one can get a sufficient condition for equality in (\ref{dew}).

\begin{pro}
If $M$ is a $2$-dimensional entire minimal graph in $\R^{2+m}$, then
\begin{equation}\label{dew1}
\De \log w=-|B|^2.
\end{equation}
\end{pro}

\begin{proof}
By Lemma \ref{iso}, there exists a nonsingular linear transformation
\begin{equation}\label{iso2}
 \left(\begin{array}{c} u^1 \\ u^2\end{array}\right)=
\left(\begin{array}{cc} 1 & 0\\ a & b\end{array}\right)
\left(\begin{array}{c} x^1 \\ x^2\end{array}\right)
\end{equation}
such that $(u^1,u^2)$ are global isothermal parameters for $M$, where
$a$ and $b>0$ are constants. Hence there is a positive function $\la$ on $M$,
such that the metric $g$ on $M$ can be expressed as
\begin{equation}g=\la^2\big((du^1)^2+(du^2)^2\big).\end{equation}
As shown in (\ref{v1}),
$$w^{-1}=v=\la^2\la_1\la_2$$
with $\la_1^2\geq \la_2^2>0$ being eigenvalues of $\left(\begin{array}{cc}
1+a^2 & ab\\
ab & b^2
\end{array}\right)$. Thus
\begin{equation*}
\log w=-\log (\la^2)-\log(\la_1\la_2)
\end{equation*}
and moreover
\begin{equation}\label{dew2}
\De \log w=-\De \log (\la^2).
\end{equation}

The Gauss curvature $K$ of $M$ is given by (see e.g. \cite{j})
\begin{equation}\label{dew3}
 K=-\f{1}{2}\De \log(\la^2).
\end{equation}

On the other hand, let $\{e_1,e_2\}$ be an orthonormal basis of $T_p M$, with $p$ an arbitrary point in $M$. Since $M$ is minimal,
$B_{e_1e_1}+B_{e_2e_2}=0$ and the Gauss equation yields
\begin{equation}\label{dew4}
K=\mathrm{det} B_{e_ie_j}=-\f{1}{2}|B|^2.
\end{equation}
Finally combining (\ref{dew2}), (\ref{dew3}), (\ref{dew4}) yields (\ref{dew1}).
\end{proof}

Let $M$ be a Riemann surface and $F=(F^1,\cdots,F^{n+m}):M\ra \R^{2+m}$ be an isomorphic immersion. Every $p\in M$ has
a coordinate neighborhood $(U;u,v)$ such that $g=\la^2(du^2+dv^2)$ on $U$. Now we introduce the complex coordinate
\begin{equation*}
 w=u+\sqrt{-1}v.
\end{equation*}
It is well-known that $F$ is minimal if and
only if all components
of $F$ are harmonic functions on $M$, i.e. $\pd{F}{w}$ is a
vector-valued holomorphic function on $U$;
here and in the sequel
\begin{equation}\label{hol}\aligned
\f{\p}{\p w}=\f{1}{2}\Big(\f{\p}{\p u}-\sqrt{-1}\f{\p}{\p v}\Big),\qquad &\pf{\bar{w}}=\f{1}{2}\Big(\pf{u}+\sqrt{-1}\pf{v}\Big),\\
dw=du+\sqrt{-1}dv,\qquad &d\bar{w}=du-\sqrt{-1}dv.
\endaligned
\end{equation}
While $\pd{F}{w}$ depends on the choice of local coordinate, the vector-valued holomorphic $1$-form
$$\p F:=\pd{F}{w}dw$$
is independent  of these local coordinates and can be well-defined on the whole surface, where $dw=du+\sqrt{-1}dv$. Similarly we can define
$\bar{\partial} F:=\pd{F}{\bar{w}}d\bar{w}.$

With the symmetric bi-linear form
$$\lan (a_1,\cdots,a_N),(b_1,\cdots,b_N)\ran=\sum_{i=1}^N a_i b_i,$$
since $(u,v)$ are isothermal parameters, it is well known
and easy to check that
\begin{equation}\label{w1}
 \Big\lan \pd{F}{w},\pd{F}{w}\Big\ran=0,\qquad \Big\lan \pd{F}{w},\pd{F}{\bar{w}}\Big\ran>0
\end{equation}
which is equivalent to
\begin{equation}
 \lan \p F,\p F\ran=0,\qquad \lan \p F,\bar{\partial} F\ran>0.
\end{equation}

 Similarly, one can define
\begin{equation}
 \p^2 F:=\ppd{F}{w}dw^2,\qquad \bar{\partial}^2 F:=\ppd{F}{\bar{w}}d\bar{w}^2.
\end{equation}
Then the minimality of $F$ implies that $\p^2 F$ is a vector-valued holomorphic $2$-form.

\begin{pro}\label{simons2}
 For a fixed point $p$ in a  minimal surface $M\subset \R^{2+m}$, the following statements are equivalent:

(a) $\De |B|^2=|\n B|^2-3|B|^4$ at $p$;

(b) $p$ is a G-conformal point;

(c) $\big\lan B_{\pf{w}\pf{w}},B_{\pf{w}\pf{w}}\big\ran=0$ at $p$, where $w$ is a local complex coordinate near $p$;

(d) $\lan \p^2 F,\p^2 F\ran=0$ at $p$.

\end{pro}

\begin{proof}

The equivalence of (a) and (b) has been proved in Proposition \ref{simons}.

Since $(u,v)$ is an isothermal coordinate, $\pf{u}$ and $\pf{v}$ have
the same length and are orthogonal to each other, hence
$p$ is an holomorphic-like point if and only if
\begin{equation}
|B_{uu}|=|B_{uv}|,\qquad \lan B_{uu},B_{uv}\ran=0.
\end{equation}
Here and in the sequel, $B_{uu}:=B_{\pf{u}\pf{u}}$, $B_{uv}:=B_{\pf{u}\pf{v}}$ and so on.

By using (\ref{hol}) one can get
\begin{equation}\label{b11}
 B_{ww}=\f{1}{2}B_{uu}-\f{\sqrt{-1}}{2}B_{uv}.
\end{equation}
It implies
\begin{equation}\aligned
 \lan B_{ww},B_{ww}\ran=\f{1}{4}\big(|B_{uu}|^2-|B_{uv}|^2\big)-\f{\sqrt{-1}}{2}\lan B_{uu},B_{uv}\ran
\endaligned
\end{equation}
and hence (b) and (c) are equivalent.

Since
$$(T_p M)\otimes \Bbb{C}=\text{span}\Big\{\pd{F}{w},\pd{F}{\bar{w}}\Big\}$$
there exist two complex numbers $\mu_1$ and $\mu_2$, such that
\begin{equation}\label{w2}
\n_{\pf{w}}\pf{w}=\Big(\ppd{F}{w}\Big)^T=\mu_1 \pd{F}{w}+\mu_2\pd{F}{\bar{w}}.
\end{equation}
By (\ref{w1}),
\begin{equation}\aligned
 0&=\f{1}{2}\pf{w}\Big\lan \pd{F}{w},\pd{F}{w}\Big\ran=\Big\lan \ppd{F}{w},\pd{F}{w}\Big\ran\\
&=\Big\lan \mu_1\pd{F}{w}+\mu_2\pd{F}{\bar{w}},\pd{F}{w}\Big\ran=\mu_2\Big\lan \pd{F}{\bar{w}},\pd{F}{w}\Big\ran.
\endaligned
\end{equation}
Hence $\mu_2=0$ and moreover
\begin{equation}
\aligned
\Big\lan \ppd{F}{w},\ppd{F}{w}\Big\ran&=\Big\lan \Big(\ppd{F}{w}\Big)^N,\Big(\ppd{F}{w}\Big)^N\Big\ran+\Big\lan \Big(\ppd{F}{w}\Big)^T,\Big(\ppd{F}{w}\Big)^T\Big\ran\\
&=\lan B_{ww},B_{ww}\ran+\mu_1^2\Big\lan \pd{F}{w},\pd{F}{w}\Big\ran=\lan B_{ww},B_{ww}\ran.
\endaligned
\end{equation}
Thus (c) is equivalent to (d).

\end{proof}

Define
\begin{equation}
 \om:=\lan \p^2 F,\p^2 F\ran=\Big\lan \ppd{F}{w},\ppd{F}{w}\Big\ran dw^4
\end{equation}
then it is easy to check that the definition of $\om$ is independent of the choice of coordinate, and
$$\pf{\bar{w}}\Big\lan \ppd{F}{w},\ppd{F}{w}\Big\ran=2\Big\lan \pf{w}\Big(\f{\p^2 F}{\p w\p \bar{w}}\Big),\ppd{F}{w}\Big\ran=0$$
implies $\om$ is a homolomorphic $4$-form on $M$. By using Proposition \ref{simons2} we immediately get the following corollary.

\begin{cor}
 Let $M$ be a minimal surface in $\ir{2+m}$, then $M$ is totally G-conformal  if and only if the holomorphic $4$-form
$\om:=\lan \p^2 F,\p^2 F\ran$ vanishes everywhere.
\end{cor}

\begin{cor}
 Let $M=\{(x,f(x)):x\in\R^2\}$ be an entire minimal graph in $\R^4$. Then $M$ is totally G-conformal if and only if at least one of the following
3 cases occurs: (i) $f:\R^2\ra \R^2$ is a holomorphic function; (ii) $f$ is anti-holomorphic; (iii) $f$ is affine linear.
\end{cor}

\begin{proof}
Let $(u_1,u_2)$ be the global isothermal parameters on $M$ given in (\ref{iso2}). Denote $z:=u_1+\sqrt{-1}u_2$ and
\begin{equation}
 \phi_i=\pd{x^i}{z},\qquad \phi_{2+\a}=\pd{f^\a}{z}.
\end{equation}
then $\pd{F}{z}=(\phi_1,\phi_2,\phi_3,\phi_4)$ and (\ref{w1}) yields $\phi_1^2+\phi_2^2+\phi_3^2+\phi_4^2=0$. By (\ref{iso2}), $\phi_1$ and $\phi_2$ are both constants, denote
\begin{equation}
 d:=\phi_1^2+\phi_2^2,
\end{equation}
then
\begin{equation}
 \phi_3^2+\phi_4^2=-(\phi_1^2+\phi_2^2)=-d.
\end{equation}

If $d=0$, then $\phi_4=\pm \sqrt{-1}\phi_3$ and hence
\begin{equation}
 \ppd{F}{z}=(\phi'_1,\phi'_2,\phi'_3,\phi'_4)=(0,0,\phi'_3,\pm \sqrt{-1}\phi'_3).
\end{equation}
It follows that
\begin{equation}
\Big\lan \ppd{F}{z},\ppd{F}{z}\Big\ran=(\phi'_3)^2-(\phi'_3)^2=0
\end{equation}
and $M$ is totally holomorphic-like. As show in \cite{h-s-v}, $d=0$ implies $f$ is holomorphic or anti-holomorphic, and vice versa.

If $d\neq 0$, then
\begin{equation}\label{hol1}
 -d=\phi_3^2+\phi_4^2=(\phi_3+\sqrt{-1}\phi_4)(\phi_3-\sqrt{-1}\phi_4)
\end{equation}
implies $\phi_3-\sqrt{-1}\phi_4$ is an entire function having no zeros, hence there is an entire function $H(z)$, such that
\begin{equation}
 \phi_3-\sqrt{-1}\phi_4=e^{H(z)}.
\end{equation}
Substituting it into (\ref{hol1}) gives
\begin{equation}
 \phi_3+\sqrt{-1}\phi_4=-de^{-H(z)}.
\end{equation}
In conjunction with the above two equations we have
\begin{equation}
 \phi_3=\f{1}{2}(e^H-d e^{-H}),\qquad \phi_4=\f{\sqrt{-1}}{2}(e^H+d e^{-H}).
\end{equation}
Thus
\begin{equation}
 \aligned
\Big\lan \ppd{F}{z},\ppd{F}{z}\Big\ran&=(\phi'_1)^2+(\phi'_2)^2+(\phi'_3)^2+(\phi'_4)^2\\
&=\f{1}{4}(e^H+d e^{-H})^2(H')^2-\f{1}{4}(e^H-d e^{-H})^2(H')^2\\
&=d(H')^2
\endaligned
\end{equation}
which is identically zero if and only if $H$ is a constant function. In this case, $\phi_i$ and $\phi_{2+\a}$ are all constants
on $M$, hence $M$ has to be an affine plane.

\end{proof}

For the sequel, we put
\begin{equation}
(\n B)_{uuv}:=(\n_{\pf{v}}B)\Big(\pf{u},\pf{u}\Big),\qquad (\n B)_{www}:=(\n_{\pf{w}}B)\Big(\pf{w},\pf{w}\Big)
\end{equation}
and so on. Then (\ref{kato2}) says that there are $\xi_1,\xi_2\in \R$, such that
\begin{equation}\label{nb10}
\aligned
(\n B)_{uuu}&=\xi_1 B_{uu}-\xi_2 B_{uv}\\
(\n B)_{uuv}&=\xi_2 B_{uu}+\xi_1 B_{uv}.
\endaligned
\end{equation}

\begin{pro}\label{p10}
 For a fixed point $p$ in a  minimal surface $M\subset \R^{2+m}$, the following statements are equivalent:

(a) $|\n B|^2=2\big|\n |B|\big|^2$ at $p$;

(b) There is an isothermal coordinate chart $(U; u,v)$ around $p$, such that $(\n B)_{www}=\ze B_{ww}$ at $p$, with $w=u+\sqrt{-1}v$ and $\ze\in \Bbb{C}$;

(c) For an arbitrary isothermal coordinate chart $(U; u,v)$ around $p$, there is $\ze\in \Bbb{C}$, such that $(\n B)_{www}=\ze B_{ww}$ at $p$, with
$w=u+\sqrt{-1}v$.

\end{pro}

\begin{proof}
The equivalence of (b) and (c) is obvious, so it is sufficient to prove the equivalence of (a) and (b).

Similarly to Section \ref{s1}, one can choose an isothermal coordinate neighborhood $(U;u,v)$ of $p$, such that
$$\lan B_{uu}, B_{uv}\ran=0\qquad \text{at }p.$$
Then by Proposition \ref{kato1}, (a) is equivalent to (\ref{nb10}).

By (\ref{hol}), one can obtain
\begin{equation}\label{nb11}
(\n B)_{www}=\f{1}{2}(\n B)_{uuu}-\f{\sqrt{-1}}{2}(\n B)_{uuv}
\end{equation}
with the aid of the Codazzi equations.
If (\ref{nb10}) holds, letting $\ze:=\xi_1-\sqrt{-1}\xi_2$ and combining with (\ref{b11}) and (\ref{nb11}) implies
\begin{equation}
\aligned
\ze B_{ww}
&=\f{1}{2}(\xi_1 B_{uu}-\xi_2 B_{uv})-\f{\sqrt{-1}}{2}(\xi_1 B_{uv}+\xi_2 B_{uu})\\
&=\f{1}{2}(\n B)_{uuu}-\f{\sqrt{-1}}{2}(\n B)_{uuv}=(\n B)_{www}.
\endaligned
\end{equation}
Conversely, if $(\n B)_{www}=\ze B_{ww}$, then by letting $\xi_1=\text{Re}\ze$ and $\xi_2=-\text{Im}\ze$, one can proceed similarly to above to
get (\ref{nb10}). Therefore (a) and (b) are equivalent.

\end{proof}

\begin{cor}
Let $M$ be a totally G-conformal minimal surface in $\R^4$, then
\begin{equation}\label{nb12}
|\n B|^2=2\big|\n |B|\big|^2
\end{equation}
holds at any $p\in M$ satisfying $|B|^2(p)>0$.
\end{cor}

\begin{proof}

Since $M$ is totally holomorphic-like, $B_{uu}$ and $B_{uv}$ have the
same length and are orthogonal to each other. Since  $\dim N_pM=2$ and
$|B|^2(p)>0$ we conclude that $N_p M=\text{span}\{B_{uu},B_{uv}\}$ and moreover
$$N_p M\otimes \C=\text{span}\{B_{ww},B_{\bar{w}\bar{w}}\}.$$
Thus there are $\mu_3,\mu_4\in \Bbb{C}$, such that
$$(\n B)_{www}=\mu_3 B_{ww}+\mu_4 B_{\bar{w}\bar{w}}.$$

Differentiating both sides of $\lan B_{ww},B_{ww}\ran=0$ yields
\begin{equation}\aligned
 0&=\f{1}{2}\pf{w}\lan B_{ww},B_{ww}\ran=\lan \n_{\pf{w}}(B_{ww}),B_{ww}\ran\\\
&=\lan (\n B)_{www}+2B_{\n_{\pf{w}}\pf{w},\pf{w}},B_{ww}\ran\\
&=\lan (\n B)_{www},B_{ww}\ran+2\mu_1\lan B_{ww},B_{ww}\ran\\
&=(\mu_3+2\mu_1)\lan B_{ww} ,B_{ww}\ran+\mu_4\lan B_{\bar{w}\bar{w}},B_{ww}\ran\\
&=\mu_4|B_{ww}|^2
\endaligned
\end{equation}
where we have used (\ref{w2}).
Hence $\mu_4=0$ and then (\ref{nb12}) follows from Proposition \ref{p10}.

\end{proof}

\bibliographystyle{amsplain}

\end{document}